\documentclass[]{theclass}
\usepackage{graphicx, xfrac, lineno, float, subcaption, tasks, comment, xcolor, booktabs, multirow}
\usepackage[normalem]{ulem}
\usepackage{datetime}
\usepackage{colortbl}
\usepackage{enumerate}

\settasks{
	counter-format=(tsk[r]),
	label-width=4ex
}

\begin{document}

%\linenumbers

\begin{frontmatter}

\titledata{The Erd\H{o}s--Faber--Lov\'{a}sz Conjecture revisited}{}

\authordatatwo{John Baptist Gauci}{john-baptist.gauci@um.edu.mt}{}{Jean Paul Zerafa}{zerafa.jp@gmail.com}{The research work disclosed in this publication is funded by the ENDEAVOUR Scholarship Scheme (Malta). The scholarship may be part-financed by the European Union -- European Social Fund (ESF) under Operational Programme II -- Cohesion Policy 2014--2020, ``Investing in human capital to create more opportunities and promote the well being of society".}{Department of Mathematics, University of Malta, Malta}


\begin{abstract}
The Erd\H{o}s--Faber--Lov\'{a}sz Conjecture, posed in 1972, states that if a graph $G$ is the union of $n$ cliques of order $n$ (referred to as defining $n$-cliques) such that two cliques can share at most one vertex, then the vertices of $G$ can be properly coloured using $n$ colours. Although still open after almost 50 years, it can be easily shown that the conjecture is true when every shared vertex belongs to exactly two defining $n$-cliques. We here provide a quick and easy algorithm to colour the vertices of $G$ in this case, and discuss connections with clique-decompositions and edge-colourings of graphs. 
\end{abstract}
\end{frontmatter}
\section{Introduction}\label{section intro}

For any graph $G$, the \emph{chromatic number} $\chi(G)$ is the least number of colours that are required to colour all the vertices of $G$ such that no two adjacent vertices receive the same colour. In $1972$, Erd\H{o}s, Faber, and Lov\'{a}sz posed the following conjecture, whilst at a tea party in Boulder Colorado.

\begin{conjecture}[EFL Conjecture] \cite{Erdos1981}  \label{EFLConjecture}
If a graph $G$ is the union of $n$ cliques of order $n$, no two of which share more than one vertex, then $\chi(G)=n$. 
\end{conjecture}
For ``a proof or disproof" of the conjecture, Paul Erd\H{o}s initially offered $50$USD, but then, having seen that the problem is not so trivial and simple as it seems, he increased his offer to $500$USD. To this day, no complete solution of this conjecture exists. We refer the reader to \cite{efl0,efl1,efl2,efl3,efl4,efl5} for a more thorough introduction to the conjecture and recent results. In particular, we remark that the fractional version of the EFL Conjecture was solved by Kahn and Seymour \cite{KahnSeymour1992} in 1992. Moreover, in January 2021, it was announced \cite{kuhn} that the conjecture is true for sufficiently large values of $n$, which to our knowledge, is the best result so far in trying to attack the EFL Conjecture. 

For every positive integer $n$, let $\mathbb{EFL}_{n}$ denote the class of graphs that are constructed as the union of $n$ cliques $Q_{1}, \ldots, Q_{n}$ each of order $n$, such that any two of these $n$ cliques intersect in at most one vertex. The $n$-cliques $Q_{1}, \ldots, Q_{n}$ are referred to as the \emph{defining} $n$-cliques, and a vertex which belongs to more than one defining $n$-clique is said to be \emph{shared}. Conjecture \ref{EFLConjecture} is equivalent to saying that if $G \in \mathbb{EFL}_{n}$, then $\chi(G)=n$. To avoid the trivial case when $n=1$, we tacitly assume that $n\geq 2$. We also remark that since a graph $G$ belonging to $\mathbb{EFL}_{n}$ contains $n$-cliques, $\chi(G)\geq n$, and so in order to prove that $\chi(G)=n$ it suffices to provide a proper $n$-colouring of the vertices of $G$.

Let $\mathcal{G}$ be in $\mathbb{EFL}_{n}$ such that every shared vertex belongs to \textbf{exactly} two defining $n$-cliques of $\mathcal{G}$. 
It is known that, in this case, $\mathcal{G}$ admits a proper $n$-vertex-colouring by using clique-decompositions and edge-colourings as described in the following (for a more detailed discussion, the reader is referred to Section \ref{section last}). In fact, consider the complete graph $K_{n}$ on $n$ vertices, with each vertex of $K_{n}$ corresponding to a defining $n$-clique of $\mathcal{G}$ and with every shared vertex of $\mathcal{G}$ corresponding to an edge of $K_{n}$. By Vizing's Theorem, the edges of $K_{n}$ can be properly coloured using at most $n$ colours, and consequently, the shared vertices of $\mathcal{G}$ can also be properly coloured using at most $n$ colours. This colouring is then extended to a proper $n$-colouring of all the vertices of $\mathcal{G}$, as follows. Every defining $n$-clique of $\mathcal{G}$ contains at most $n-1$ shared vertices, which by the above are all coloured differently. Let $\mathcal{C}_i$ denote the set of colours of the shared vertices belonging to the defining $n$-clique $Q_{i}$. Since the number of unshared vertices of $Q_{i}$ is equal to $n-|\mathcal{C}_i|$, the unshared vertices of $Q_{i}$ can be assigned mutually different colours from $\{\,1,\ldots, n\,\}-\mathcal{C}_i$, yielding a proper $n$-vertex-colouring of $\mathcal{G}$. Consequently, $\chi(\mathcal{G})=n$, and hence, the EFL Conjecture holds for this particular instance. 

The aim of this note is to present a very simple and straightforward algorithm that enables the construction of a proper colouring of the vertices of $\mathcal{G}$ with $n$ colours, when every shared vertex belongs to exactly two defining $n$-cliques, without having the need to first colour the edges of the corresponding complete graph $K_n$ and then transferring this colouring to $\mathcal{G}$ (as described above). 
We first consider the special case when the number of shared vertices in $\mathcal{G}$ achieves the maximum value ${n \choose 2}$, that is, when every two defining $n$-cliques share a vertex and every shared vertex belongs to exactly two defining $n$-cliques. Lemma \ref{lemma main} gives the algorithm for the case when $n$ is even, and then with the help of Example \ref{Example EFL mod}, this is generalised to the case when $n$ is odd (see Proposition \ref{prop EFL algorithm1}). Finally, in Theorem \ref{Theorem EFL algorithm2}, we discuss how this algorithm can be utilised in the case when the number of shared vertices is less than $\binom{n}{2}$. In fact, our algorithm properly colours the shared vertices of a graph $\mathcal{G}\in\mathbb{EFL}_{n}$ (in which every shared vertex belongs to exactly two defining $n$-cliques) using at most $n$ colours. One can then extend this colouring to a proper $n$-colouring of all the vertices of $\mathcal{G}$, as explained above. In Section \ref{section last} we discuss a very natural reformulation of the EFL Conjecture, suggested in \cite{efl0}, in terms of clique-decompositions and edge-colourings. We end this note by suggesting Problem \ref{problem} which although we believe is captivating in itself, we think it can shed some further light on an eventual complete solution of the EFL Conjecture.

Although most of our terminology is standard, we refer the reader to \cite{BM} for further definitions and notation not explicitly stated.

\section{Main result}\label{section main}
For $i$ and $j$ in $\{\,1,\ldots,n\,\}$ and $i<j$, if the defining $n$-cliques $Q_{i}$ and $Q_{j}$ have a shared vertex, then this shared vertex is denoted by the ordered pair $(i,j)$. Let $\mathcal{G}_{n}$ be the unique graph in $\mathbb{EFL}_{n}$ having $\binom{n}{2}$ shared vertices, and let $\mathcal{V}_{n}\subset V(\mathcal{G}_{n})$ be the set of shared vertices of $\mathcal{G}_{n}$, that is, $\vert \mathcal{V}_{n}\vert=\binom{n}{2}$. In this case every shared vertex belongs to exactly two defining $n$-cliques. We remark that in what follows the complete residue system used when taking operations modulo $t$ is $\{\,1,\ldots, t\,\}$. We first consider the case when $n$ is an even integer.

\begin{lemma}\label{lemma main}
Let $n\geq 2$ be an even integer. The function $c:\mathcal{V}_{n}\rightarrow \{\,1,\ldots, n-1\,\}$, such that 
\begin{linenomath}
$$c\big((i,j)\big)\equiv\left\{
\begin{array}{rl}
i+j\pmod{n-1} & $if $j<n,\\
2i\pmod{n-1} & $if $j=n,
\end{array}\right.
$$
\end{linenomath}
is a proper $(n-1)$-colouring of the vertices in $\mathcal{V}_{n}$. 
\end{lemma}
\begin{proof}
Let $(i,j)$ and $(k,\ell)$ be two adjacent vertices. It suffices to show that $c\big((i,j)\big)\neq c\big((k,\ell)\big)$. For contradiction, we suppose that $c\big((i,j)\big)= c\big((k,\ell)\big)$. We first note that exactly one of $i$ and $j$ has to be equal to exactly one of $k$ and $\ell$. There are five cases that need to be considered.\medskip

\noindent\textbf{Case 1.} $i=k$ and either $j<\ell<n$ or $\ell<j<n$.

Notwithstanding whether $j<\ell<n$ or $\ell<j<n$, we have $i+j\equiv k+\ell\equiv i+\ell\pmod{n-1}$, implying that $j\equiv \ell\pmod{n-1}$, a contradiction.\,\,\,{\tiny$\blacksquare$} \medskip

%Without loss of generality, let $j<\ell$. Then, $i+j\equiv k+\ell\equiv i+\ell\pmod{n-1}$. Thus, for some integer $\lambda_{1}\geq 1$, we have $i+\ell=i+j+\lambda_{1}(n-1)$, implying that $\ell-j=\lambda_{1}(n-1)$. However, $2\leq j<\ell\leq n-1$, which gives $n-3\geq \lambda_{1}(n-1)$, a contradiction.\,\,\,{\tiny$\blacksquare$}

\noindent\textbf{Case 2.} $j=\ell<n$ and either $i<k$ or $k<i$.

Notwithstanding whether $i<k$ or $k<i$, we have $i+j\equiv k+\ell \equiv k+j\pmod{n-1}$, implying that $i\equiv k\pmod{n-1}$, a contradiction.\,\,\,{\tiny$\blacksquare$} \medskip

%Without loss of generality, let $i<k$. Then, $i+j\equiv k+\ell \equiv k+j\pmod{n-1}$. Thus, for some integer $\lambda_{2}\geq 1$, we have $k+j=i+j+\lambda_{2}(n-1)$, implying that $k-i=\lambda_{2}(n-1)$. However, $1\leq i<k\leq n-2$, which gives $n-3\geq \lambda_{2}(n-1)$, a contradiction.\,\,\,{\tiny$\blacksquare$}

\noindent\textbf{Case 3.} Either $j=k$ or $i=\ell$.

Without loss of generality, let $j=k$. Then, $i<\ell\leq n$ and $j<n$. Assuming first that $\ell<n$, we get that $i+j\equiv k+\ell\equiv j+\ell\pmod{n-1}$, implying that $i\equiv j\pmod{n-1}$, a contradiction. Thus, $\ell=n$, and $i+j\equiv 2k\equiv 2j\pmod{n-1}$, implying that $i\equiv j\pmod{n-1}$, a contradiction once again.\,\,\,{\tiny$\blacksquare$} \medskip

%Without loss of generality, let $j=k$. Then, $i<\ell\leq n$ and $j<n$. Assuming first that $\ell<n$, we get that $i+j\equiv k+\ell\equiv j+\ell\pmod{n-1}$. Thus, for some integer $\lambda_{3}\geq 1$, we have $j+\ell=i+j+\lambda_{3}(n-1)$, implying that $\ell-i=\lambda_{3}(n-1)$. However, $1\leq i<\ell\leq n-1$, which gives $n-2\geq \lambda_{3}(n-1)$, a contradiction. Thus, $\ell=n$, and $i+j\equiv 2k\equiv 2j\pmod{n-1}$. Thus, for some integer $\lambda_{3}'\geq 1$, we have $2j=i+j+\lambda_{3}'(n-1)$, implying that $j-i=\lambda_{3}'(n-1)$. However, $1\leq i<j\leq n-1$, which gives $n-2\geq \lambda_{3}'(n-1)$, a contradiction.\,\,\,{\tiny$\blacksquare$}

\noindent\textbf{Case 4.} $j=\ell=n$ and either $i<k$ or $k<i$.

Notwithstanding whether $i<k$ or $k<i$, we have $2i\equiv 2k\pmod{n-1}$, and since $n$ is even, this implies that $i\equiv k\pmod{n-1}$, a contradiction.\,\,\,{\tiny$\blacksquare$} \medskip

%Without loss of generality, let $i<k$. Then, $2i\equiv 2k\pmod{n-1}$. Thus, $2k=2i+\lambda_{4}(n-1)$, for some integer $\lambda_{4}\geq 1$. Since, $1\leq i<k\leq n-1$, we have $2(n-2)\geq 2(k-i)=\lambda_{4}(n-1)$, implying that $\lambda_{4}$ must be equal to 1. Therefore, $2(k-i)=n-1$, implying that $n-1$ is even, a contradiction.\,\,\,{\tiny$\blacksquare$}

\noindent\textbf{Case 5.} $i=k$ and either $j<\ell=n$ or $\ell<j=n$.

Without loss of generality, let $j<\ell=n$. Then, $i+j\equiv 2k\equiv 2i\pmod{n-1}$, implying that $j\equiv i\pmod{n-1}$, a contradiction.\,\,\,{\tiny$\blacksquare$} \medskip

%Without loss of generality, let $j<\ell=n$. Then, $i+j\equiv 2k\equiv 2i\pmod{n-1}$. Thus, $i+j=2i+\lambda_{5}(n-1)$, for some integer $\lambda_{5}\geq 1$ because $i<j$. However this implies that $n-1\geq j=i+\lambda_{5}(n-1)>n-1$, a contradiction.\,\,\,{\tiny$\blacksquare$}

Hence $c$ is a proper $(n-1)$-colouring of the shared vertices of $\mathcal{G}_{n}$.
\end{proof}

Since every defining $n$-clique of $\mathcal{G}_{n}$ contains $n-1$ shared vertices, it can be easily seen that the proper $(n-1)$-vertex-colouring given in Lemma \ref{lemma main} can be extended to a proper $n$-vertex-colouring of $\mathcal{G}_{n}$, by assigning a unique new colour to all the vertices in $V(\mathcal{G}_{n})-\mathcal{V}_{n}$, implying that $\chi(\mathcal{G}_{n})=n$. 

\begin{example}\label{Example EFL mod}
Here, we use the colouring explained in Lemma \ref{lemma main} to obtain a proper $9$-colouring of the shared vertices of $\mathcal{G}_{10}$. Note that addition is taken modulo $9$.

\begin{table}[H]
\centering
%\caption{}
%\label{my-label}
\begin{tabular}{ccccccccc}
\addlinespace[-\aboverulesep] 
\cmidrule[\heavyrulewidth]{1-9}%\toprule
\multicolumn{9}{c}{\emph{Colours}}                                                                                                           \\ 
1               & 2      & 3               & 4               & 5               & 6               & 7               & 8      &9        \\ 
\cmidrule{1-9}
%\multirow{5}[0]{*}{$\mathcal{V}^{10}$}
(1,9)           & (5,6) & (1,2)           & (1,3)           & (1,4)           & (1,5)           & (1,6)           & (1,7)  & (1,8)                    \\
(2,8)           & (2,9)  & (3,9)           & (4,9)           & (2,3)           & (2,4)           & (2,5)           & (2,6)    &(2,7)        \\
(3,7)           & (3,8)  & (4,8)           & (5,8)           & (5,9)           & (6,9)           & (3,4)           & (3,5)   & (3,6)         \\
(4,6)           & (4,7)  & (5,7)           & (6,7)           & (6,8)           & (7,8)           & (7,9)           & (8,9)  &(4,5)                  \\
\textbf{(5,10)} & \textbf{(1,10)}  & \textbf{(6,10)} & \textbf{(2,10)} & \textbf{(7,10)} & \textbf{(3,10)} & \textbf{(8,10)} & \textbf{(4,10)}&\textbf{(9,10)}  \\ \cmidrule[\heavyrulewidth]{1-9}
\end{tabular}
\caption{A proper $9$-colouring of the shared vertices in $\mathcal{G}_{10}$}
\end{table}

Colour $10$ is then given to the vertices in $V(\mathcal{G}_{10})-\mathcal{V}_{10}$.
\end{example}

Note that if we remove the shared vertices marked bold, the above table reduces to a proper $9$-colouring of the shared vertices of $\mathcal{G}_{9}$. In fact, let $n\geq 3$ be an odd integer. The above idea can be used for a proper $n$-colouring of the shared vertices of $\mathcal{G}_{n}$. More precisely, by the algorithm given in Lemma \ref{lemma main}, we know that the colouring 
\begin{linenomath}
$$c\big((i,j)\big)\equiv\left\{
\begin{array}{rl}
(i+j)\pmod{n} & $if $j<n+1,\\
2i\pmod{n} & $if $j=n+1,
\end{array}\right.
$$
\end{linenomath}
provides a proper $n$-colouring of the shared vertices of $\mathcal{G}_{n+1}$, since $n+1$ is even. Restricting the above colouring to those shared vertices $(i,j)$ for which $j\leq n$, we get that $c\big((i,j)\big)\equiv i+j\pmod{n}$ is a proper $n$-colouring of the shared vertices of $\mathcal{G}_{n}$. As before, this gives a proper $n$-colouring of the vertices in $\mathcal{G}_{n}$, implying that $\chi(\mathcal{G}_{n})=n$. The next proposition follows immediately from the above arguments.

\begin{proposition}\label{prop EFL algorithm1} If $n$ is even, then 
\begin{linenomath}
$$c\big((i,j)\big)\equiv\left\{
\begin{array}{rl}
i+j\pmod{n-1} & $if $j<n\\
2i\pmod{n-1} & $if $j=n,
\end{array}\right.
$$
\end{linenomath}
is a proper $(n-1)$-colouring of the shared vertices of $\mathcal{G}_{n}$. Otherwise, if $n$ is odd, $c\big((i,j)\big)\equiv i+j\pmod{n}$ is a proper $n$-colouring of the shared vertices of $\mathcal{G}_{n}$. \qed
\end{proposition}

Now, let $\mathcal{G}\in\mathbb{EFL}_{n}$ for some integer $n$, such that every shared vertex belongs to exactly two defining $n$-cliques. Let $\mathcal{V}$ be the set of all the shared vertices of $\mathcal{G}$. Then, $\mathcal{V}\subseteq \mathcal{V}_{n}$, and so $\vert \mathcal{V}\vert \leq \binom{n}{2}$. Consider the case when $\vert \mathcal{V}\vert < \binom{n}{2}$. The colouring $c$ given in Proposition \ref{prop EFL algorithm1} is also a proper colouring of the shared vertices of $\mathcal{G}$ using at most $n$ colours, since every shared vertex of $\mathcal{G}$ corresponds to a shared vertex of $\mathcal{G}_{n}$. Hence, the algorithm described in Proposition \ref{prop EFL algorithm1} can be used to provide a proper colouring of the shared vertices of $\mathcal{G}$ using at most $n$ colours. As before, this can then be extended to a proper $n$-colouring of all the vertices of $\mathcal{G}$. We can thus summarise the above results in the following theorem.

\begin{theorem}\label{Theorem EFL algorithm2} Let $\mathcal{G}\in\mathbb{EFL}_{n}$ such that every shared vertex of $\mathcal{G}$ belongs to exactly two defining $n$-cliques of $\mathcal{G}$. 
\begin{enumerate}[(i)]
\item If $n$ is even, then 
\begin{linenomath}
$$c\big((i,j)\big)\equiv\left\{
\begin{array}{rl}
i+j\pmod{n-1} & $if $j<n\\
2i\pmod{n-1} & $if $j=n.
\end{array}\right.
$$
\end{linenomath}
is a proper $(n-1)$-colouring of the shared vertices of $\mathcal{G}$. 
\item If $n$ is odd, then $c\big((i,j)\big)\equiv i+j\pmod{n}$ is a proper $n$-colouring of the shared vertices of $\mathcal{G}$. 
\item $\chi(\mathcal{G})=n$.
\end{enumerate}
\end{theorem}

\section{Clique-decompositions and edge-colourings}\label{section last}

We would like to end this note by recalling that, as also indicated above in Section \ref{section intro}, the EFL Conjecture can be restated in a very simple and intuitive way in terms of clique-decompositions and edge-colourings of the complete graph. As far as we know this re-statement of the EFL Conjecture was first suggested in \cite{efl0}. Let $H$ be a simple graph on $n$ vertices. A \emph{clique-decomposition} of $H$ is a collection $\mathcal{D}=\{\,D_{1},\ldots, D_{k}\,\}$ of subgraphs of $H$, such that each $D_{i}$ is a clique, and each edge of $H$ belongs to exactly one clique from $\mathcal{D}$. We denote a clique-decomposition $\mathcal{D}$ of a graph $H$ as $(H,\mathcal{D})$. A \emph{$n$-colouring of $(H,\mathcal{D})$} is an assignment of $n$ colours to the elements of $\mathcal{D}$ such that if $V(D_{i})\cap V(D_{j})\neq\emptyset$, for some $i\neq j$, then the colours of $D_{i}$ and $D_{j}$ are distinct. One can easily visualise this as an edge-colouring (not necessarily proper) of $H$ in which the edges in each $D_{i}$ are monochromatic, and if for some $i\neq j$, $V(D_{i})\cap V(D_{j})\neq\emptyset$, then the edges of $D_{i}$ have a different colour than the edges of $D_{j}$. 

Let $G\in \mathbb{EFL}_{n}$ (not necessarily with every shared vertex belonging to exactly two defining $n$-cliques), and let $H$ be the graph on $n$ vertices, say $v_{1},\ldots, v_{n}$, with edge-set $\{\,v_{i}v_{j}\mid V(Q_{i})\cap V(Q_{j})\neq\emptyset, \textrm{ for } i\neq j\,\}$. We consider the following clique-decomposition of $H$. Let $\{\,u_{1}, \ldots, u_{k}\,\}$ be the set of shared vertices of $G$, and, for every $t\in\{\,1,\ldots, k\,\}$, let $\mathcal{I}_{t}$ be the set of indices of all the defining $n$-cliques containing $u_{t}$. Furthermore, for each $t\in\{\,1,\ldots, k\,\}$, we let $D_{t}$ be the subgraph of $H$ induced by the vertices $\{\,v_{i}\mid i\in\mathcal{I}_{t}\,\}$. Consequently, the set $\mathcal{D}=\{\,D_{1}, \ldots, D_{k}\,\}$ is a clique-decomposition of $H$. If there exists a $n$-colouring of $(H,\mathcal{D})$, then there exists a vertex colouring of the shared vertices of $G$ using $n$ colours, implying that $\chi(G)=n$. 

In general, one can deduce that every graph in $\mathbb{EFL}_{n}$ gives rise to a simple graph on $n$ vertices with a clique-decomposition, and by a similar argument, every simple graph on $n$ vertices with a clique-decomposition corresponds to some graph in $\mathbb{EFL}_{n}$. The case considered in Section \ref{section main} corresponds to the case when every clique in $\mathcal{D}$ is a $2$-clique. Moreover, as in the previous section, if one can show that for every possible clique-decomposition $\mathcal{D}$ of $K_{n}$, $(K_{n},\mathcal{D})$ admits a $n$-colouring, then the EFL Conjecture would be true. Although a proof of the conjecture for all sufficiently large values of $n$ was recently announced \cite{kuhn}, we still believe that such a problem deserves to be studied further and solved for the other instances as well, as this could give insights into related areas such as clique-decompositions and edge-colourings of graphs, which have been already studied such as in \cite{clique1,clique2}. In this sense, we suggest the following problem which we think could be a possible way forward.

\begin{problem}\label{problem}
Let $\mathcal{D}$ be a clique-decomposition of $K_{n}$, such that every clique in $\mathcal{D}$ is either a 2-clique or a $r$-clique, for some fixed $r\in\{\,3,\ldots, n-1\,\}$. Determine whether $(K_{n},\mathcal{D})$ has a $n$-colouring, and, if in the affirmative, whether an efficient algorithm to find such a $n$-colouring exists.
\end{problem}

%\begin{acknowledgements}
%  Acknowledgements to persons or institutions should be put in a
%  separate environment before the beginning of the bibliography, with
%  the following syntax:
%\begin{verbatim}
%\begin{acknowledgements}
%..........
%\end{acknowledgements}
%\end{verbatim}
%\end{acknowledgements}

\end{document}